\documentclass{amsart}
\usepackage{amsmath,amsthm,amsfonts,amssymb,latexsym,cancel}
\usepackage{amsrefs}
\usepackage{thmtools}

\newcommand{\define}[1]{\textbf{#1}} 

\declaretheorem[name={Theorem},numberwithin=section]{thm}
\declaretheorem[name={Proposition}, sibling=thm]{proposition}
\declaretheorem[name={Lemma}, sibling=thm]{lemma}

\declaretheorem[name={Claim}, numberwithin=thm]{claim}
\declaretheorem[name={Definition}, style=definition, sibling=thm]{definition}

\declaretheorem[name={Corollary}, sibling=thm]{corollary}

\newcommand{\R}{\mathbb{R}}
\newcommand{\F}{\mathbb{F}}

\newcommand{\cantor}{2^\omega}
\newcommand{\concat}{{}^\smallfrown}
\newcommand{\from}{\colon}

\DeclareMathOperator{\diam}{diam}
\DeclareMathOperator{\id}{id}

\newcommand{\inters}{\cap}
\newcommand{\biginters}{\bigcap}

\renewcommand{\subset}{\subseteq}
\renewcommand{\supset}{\supseteq}
\newcommand{\nsupset}{\not \supseteq}

\newcommand{\disjointunion}{\sqcup}
\newcommand{\nbhd}[1]{[#1]}

\begin{document}

\title{Hausdorff dimension and countable Borel equivalence relations}
\author[Marks]{Andrew Marks}
\author[Rossegger]{Dino Rossegger}
\author[Slaman]{Theodore Slaman}

\address[Rossegger]{Institute of Discrete Mathematics and Geometry, Technische Universit\"at Wien}
\email{dino.rossegger@tuwien.ac.at}

\address[Marks, Slaman]{Department of Mathematics, University of California,
Berkeley}
\email{marks@math.berkeley.edu}
\email{slaman@math.berkeley.edu}

\thanks{This paper was inspired by discussions of the authors with Uri
Andrews and Luca San Mauro at the American Institute of Mathematics
workshop ``Invariant descriptive computability theory''. The authors wish
to thank both Andrews and San Mauro for their input and the American
Institute of Mathematics for providing the venue for these discussions. The
work of the first author was partially supported by NSF grant DMS-2348788.
The
work of the second author was supported by the European Union’s Horizon 2020
Research and Innovation Programme under the Marie Skłodowska-Curie grant
agreement No. 101026834 - ACOSE and the Austrian Science Fund (FWF) 10.55776/P36781.} 
\subjclass{03E15,28A78}
\begin{abstract}
  We show that if $E$ is a countable Borel equivalence relation on $\R^n$,
  then there is a closed subset $A \subset [0,1]^n$ of Hausdorff dimension
  $n$ so that $E \restriction A$ is smooth. More generally, if $\leq_Q$ is
  a locally countable Borel quasi-order on $\cantor$ and $g$ is any gauge
  function of lower order than the identity, then there is a closed set $A$
  so that $A$ is an antichain in $\leq_Q$ and $H^g(A) > 0$. 
\end{abstract}
\maketitle

\section{Introduction}

Descriptive set theory has
provided a general setting for comparing the relative difficulty of 
classification problems in mathematics, formalized as the study of Borel reducibility
among equivalence relations. 
If $E$ and $F$ are equivalence relations on standard Borel spaces $X$
and $Y$, say that $E$ is \define{Borel reducible} to $F$ 
if there is a Borel function $f \from X \to Y$ such that for
all $x_0, x_1 \in X$, we have $x_0 \mathrel{E} x_1 \iff f(x_0) \mathrel{F} f(x_1)$.
Especially interesting in the theory are
non-classifiability results. If 
$E \nleq_B F$, then there does not exist any concretely definable (i.e.
Borel) way to use elements of $F$ as invariants to classify $E$. For
example, Hjorth's theory of turbulence has given a general tool for proving
such non-classifiability results, showing that many natural equivalence
relations in mathematics cannot be classified by the isomorphism relation
of any type of countable structure \cite{H}.


A well-studied subclass of Borel equivalence relations are the \define{countable
Borel equivalence relations}, meaning those whose equivalence classes are all
countable. To date, all known non-trivial results showing that $E \nleq_B F$ for countable Borel
equivalence relations $E$ and $F$ use measure theoretic techniques and
Borel probability measures. See for example, the 
cocycle rigidity results used to prove non-reducibility results in \cite{AK} and \cite{T}. An important problem in the theory of countable Borel equivalence
relations is to find new tools beyond just Borel probability measures for proving
non-reducibility results. Such new tools seem to
be needed to solve many open questions in the subject like the problem of
whether every countable Borel equivalence relation is Borel bounded
\cite{BJ}, whether every amenable countable Borel equivalence relation is
hyperfinite \cite[6.2.(B)]{JKL}, the increasing union problem for
hyperfinite Borel equivalence relations \cite[p 194]{DJK}, or the
universal vs measure universal problem \cite[Question 3.13]{MSS}. Those
questions are all known to have positive answers modulo a null set with
respect to any Borel probability measure, but we suspect these questions to
have negative answers in general.

There are several promising candidates for new tools that could prove such non-reducibility results
such as Martin's conjecture \cite{DK}, forcing \cite{Sm}, the
$\mathcal{L}_{\omega_1,\omega}$ model theory of countable structures
\cite{CK}, and the study of topological realizations of countable Borel
equivalence relations \cite{FKSV}. There are also several results showing
certain tools \emph{cannot} prove new non-reducibility results. These
results are often in the context where this tool has an associated
$\sigma$-ideal $I$, and we show that every countable Borel equivalence
relation becomes simple after discarding a set in this ideal, or after
restricting to an $I$-positive set. For example, we have the following
well-known theorem of generic hyperfiniteness:

\begin{thm}[Hjorth-Kechris, Sullivan-Weiss-Wright, Woodin (see {\cite[Theorem
12.1]{KM}})]
  \label{thm:generic_hyp}
  If $E$ is a countable Borel equivalence relation on a Polish space $X$,
  then there is a comeager invariant Borel set $C \subset X$ so that $E
  \restriction C$ is hyperfinite.
\end{thm}

Recall here that a countable Borel equivalence relation $E$ is hyperfinite
if and only if $E \leq_B E_0$, where $E_0$ is the equivalence relation of
eventual equality on infinite binary sequences~\cite[Theorem 7.1]{DJK}. So
no simple Baire category argument can be used to prove non-hyperfiniteness
results that $E \nleq_B E_0$ for any countable Borel equivalence relation
$E$.

%

We have an analogous result to generic hyperfiniteness in the context of
the ideal of Ramsey null subset of $[\omega]^\omega$, except that we only
have hyperfiniteness on an $I$-positive set for the Ramsey null ideal:

\begin{thm}[Mathias and Soare \cites{M,So} (see {\cite[Theorem 8.17]{KSZ}})]
  \label{thm:ramsey_generic_hyp}
  If $E$ is a countable Borel equivalence relation on $[\omega]^\omega$, then
  there is an $A \in [\omega]^\omega$ so that $E \restriction
  [A]^\omega$ is hyperfinite. 
\end{thm}


Recently, Panagiotopoulos and Wang have similarly analyzed the dual Ramsey
ideal:

\begin{thm}[{\cite[Theorem 1.2]{PW}}]
  \label{thm:dual_ramsey_generic_hyp}
  If $E$ is a countable Borel equivalence relation on $(\omega)^\omega$,
  then there is an $A \in (\omega)^\omega$ so that $E \restriction
  (A)^\omega$ is smooth. 
\end{thm}

Recall here that a Borel equivalence relation $E$ is \define{smooth} if $E
\leq_B =_\R$ where $=_\R$ is the equivalence relation of equality on $\R$.
Similar canonization theorems to the above are also known for certain other Ramsey-type
ideals by work of Kanovei-Sabok-Zapletal \cite[Theorem 8.1]{KSZ}.

The present paper investigates whether Hausdorff measures and Hausdorff dimension
can be used to prove new non-reducibility results between Borel equivalence
relations. We know that Lebesgue
measure on $\cantor$ can be used to prove many interesting non-Borel-reducibility results
(such as Slaman and Steel's proof \cite{SS} that Turing equivalence on $\cantor$
is not hyperfinite, or the result that the shift action of $\F_2$ on
$2^{\F_2}$ is not
hyperfinite \cite{K91}). 
If we take $s$-dimensional Hausdorff measure on $\cantor$ for $s < 1$,
as $s \to 1$, these measures ``approach'' Lebesgue measure. More
generally, we can take arbitrary gauge measures for gauge functions $g$ with
$\lim_{t \to 0} g(t)/t = \infty$, and let $g$ approach the identity function $g(t) = t$
which corresponds to the case of Lebesgue measure. Our hope was
that the spectrum of complexities of Borel equivalence relations that can
be ``seen'' by these $s$-dimensional
Hausdorff measures or gauge measures becomes more and more complex as $s \to
1$.

Unfortunately, this is not the case. Our main theorem shows that any gauge
measure $H^g$ with the above-mentioned property trivializes every countable Borel equivalence
relation to be smooth on a set of positive $H^g$-measure. So our main result is
another in the line of work of Theorems~\ref{thm:generic_hyp}, \ref{thm:ramsey_generic_hyp}, and \ref{thm:dual_ramsey_generic_hyp}.

\begin{thm}\label{thm:bigthm}
Suppose $g \from [0,\infty) \to [0,\infty)$ is a gauge function of lower order
than the identity and that $E$ is a countable Borel equivalence relation on $\cantor$. 
Then there is a closed set $A \subset \cantor$ such that $E \restriction A$ is
smooth, and $H^g(A) > 0$. In particular, there is a closed set of Hausdorff
dimension $1$ such that $E \restriction A$ is smooth. 
\end{thm}

We note that in contrast, the arguments of \cite{SS} and \cite{K91} show
that Turing equivalence or the orbit equivalence relation of the shift
action of $\F_2$ on $2^{\F_2}$ are both non-hyperfinite on any positive
measure set with respect to Lebesgue measure. 

By using an appropriate bijection between $\cantor$ and $[0,1]^n$ we also show
that every countable Borel equivalence relation on $\R^n$ is smooth on a set of
Hausdorff dimension $n$ (Corollary~\ref{cor:unitinterval}). 

We also prove some generalizations of these results to locally countable Borel
quasi-orders on $\R^n$ and $2^\omega$. For example, if $\leq_Q$ is any
locally countable Borel quasi-order on $2^\omega$, then there is a closed set $A \subset \cantor$ so that $A$ is an antichain under $\leq_Q$, and $A$ has Hausdorff dimension $1$. 

\section{Preliminaries}
A quasi-order $\leq_Q$ on a space $X$ is a reflexive transitive relation on
$X$. We say that $\leq_Q$ is \define{locally countable} if for every $y \in
X$, $\{x \in X \colon x \leq_Q y\}$ is countable. We say that $\leq_Q$ is
Borel if is Borel as a subset of $X^2$. 
Among the examples of locally countable Borel quasi-orders are countable Borel
equivalence relation -- equivalence relations on $X$ whose classes are
all countable. 
A reference for the theory of countable Borel equivalence relations and locally
countable Borel quasi-orders is the recent survey paper~\cite{K24} of
Kechris. Note that by 
Lusin-Novikov uniformization \cite[18.10]{K95}, if $\leq_Q$ is a countable
Borel quasi-order, then there are countably many Borel functions $(f_i:X\to
X)_{i\in\omega}$ so that $y \leq_Q x$ if and only if and only if there exists an $i \in \omega$ so that
$f_i(x) = y$. 

Our conventions surrounding Hausdorff dimension and gauge measures 
follow those of Rogers~\cite{R}. Recall that gauge measures generalize the
idea of Hausdorff measures and Hausdorff dimension to arbitrary gauge
functions. A \define{gauge function} $g
\colon [0,\infty) \to [0,\infty)$ is an
increasing function that is continuous on the right, $g(0) = 0$, and $g(t)
> 0$ for $t > 0$. If $(X,d)$ is a metric space, then 
recall that we define the $g$-measure $H^g$ on subsets of $X$ as follows:
For every $\delta > 0$, let 
\[H^g_\delta(A)=\inf\{\sum_{i =
0}^\infty g(\diam(U_i)) \colon (U_i) \text{ is an open cover of $A$ by sets
of diameter $<\delta$}\}.\]
Then the \define{$g$-measure} $H^g$ is defined as $\lim_{\delta\to 0^+}
H^g_\delta$.
\begin{definition}
  Suppose that $f$ and $g$ are gauge functions. We write $f\prec g$ if $\lim_{t\to 0^+} g(d)/f(d)=0$ (or equivalently $\lim_{t\to 0^+} f(d)/g(d)=\infty$) and say that $g$ has \define{higher order} than $f$.
\end{definition}
Below, we work with gauge measures on the Cantor space $2^\omega$ of infinite binary sequences equipped with the metric
$d(x,y) = 2^{-n}$ where $n$ is least such the $n$th bit of $x$ and $y$ differ: $x(n) \neq y(n)$. We will also
work with the spaces $\R^n$ with the Euclidean metric. In a metric space,
we let $B_r(x)$ denote the open ball of radius $r$ around a point $x$. 

The $s$-dimensional Hausdorff measure is
the gauge measure given by the power functions $g(t) = t^s$. Here if $g(t) = t$,
then $H^g$ is Lebesgue measure on $2^\omega$, and if $g(t) = t^n$, then $H^g$ is
Lebesgue measure on $\R^n$. We will often write $H^s$ for the $s$-dimensional
Hausdorff measure for $s\in \mathbb R^+$. Then the \define{Hausdorff dimension} of a set $A$ is 
\[ \dim(A)=\inf \{ s: H^s(A)=0\}=\sup\{ s: H^s(A)=\infty\}.\]

We let $2^{< \omega}$ denote the set of finite binary strings, and we use the
letters $s,t$ for its elements. We let $|s|$ denote the length of $s$ and $s(n)$
is the $n$th bit of $s$. Finally, if $s, t \in 2^{< \omega}$, we let $s
\concat t$ denote the concatenation of $s$ and $t$. 

\section{Proof of the main theorem}

First, we fix notation for describing a binary tree $T$ where at each level,
either all nodes at this level \define{split} (i.e. have two successors in $T$), or all
nodes at this level have exactly one successor in $T$.

\begin{definition}
Given a set $A \subset \omega$, and a function $y \from 2^{< \omega} \to 2$,
let $T_{A,y} \subset 2^{< \omega}$ be the set of $t \in 2^{< \omega}$ such
that for all $n < |t|$, if $n \in A$, then $t(n) = y (t \restriction
n)$.  
\end{definition}
That is, $T_{A,y}$ is the tree where if $t \in T_{A,y}$ and $|t| \notin
A$, then both $t \concat 0$ and $t \concat 1$ are in $T_{A,y}$. However, if
$|t| \in A$, then the only successor of $t$ in $T_{A,y}$ is $t \concat
y(t)$.

We also fix notation for the uniform measure on $[T_{A,y}]$: 
\begin{definition}
Let $\mu_{A,y}$ be the uniform measure on $[T_{A,y}]$, so that if $t \in
T_{A,y}$ is a splitting node, then both its successors have equal
measure, i.e., $\mu_{A,y}(\nbhd{t \concat 0}) = \mu_{A,y}(\nbhd{t \concat 1})$.
\end{definition}
Note that since all nodes in $T_{A,y}$ at a given level are either
splitting nodes, or none are splitting nodes, this implies that if $s,t \in
[T_{A,y}]$ have the same length, then $\mu_{A,y}(\nbhd{s}) =
\mu_{A,y}(\nbhd{t})$,
and indeed if $t \in T_{A,y}$ has length $n$, then $\mu_{A,y}(\nbhd{t}) = 2^{-n
+ |A \inters n|}$, since $|A \inters n|$ gives the number of non-splitting
levels below $n$. 

Our first lemma relates the rate at which elements appear in a set $A\subseteq \omega$ with the rate of convergence of gauge functions $g$ such that all $\mu_{A,y}$-positive subsets of $[T_{A,y}]$ have positive $g$-measure $H^g$.
\begin{lemma}\label{lem:sparse_measure}
  Suppose $g$ is a gauge function with $g\prec \id$ and that $A\subseteq \mathbb N$ is such that 
  \begin{equation}\label{eq:sparsityofA}\tag{$\dagger$} |
  A \cap n| \leq \log_2 \left( \frac{g(2^{-n})}{2^{-n}} \right) \end{equation} 
  for all but finitely many $n$. Then for all $y \from 2^{< \omega}
  \to 2$, and all $B \subset [T_{A,y}]$ with $\mu_{A,y}(B)
  > 0$, we have $H^g(B)>0$.
\end{lemma}
\begin{proof}
  Our proof relies on the following claim, which is essentially one direction of
  Frostman's lemma.
  \begin{claim}
    Suppose $\mu$ is a Borel probability measure on $2^\omega$ such that
    for all $x \in 2^\omega$ and sufficiently small $r > 0$, $g(r) > \mu(B_r(x))$. 
    Then $\mu(B)>0$ implies $H^g(B) > 0$.
  \end{claim}
  \begin{proof}
    Consider an open cover $(U_i)$ of $x$ by sets of sufficiently small diameter $r$. We
    may assume the cover is by open balls $U_i = B_{r_i}(x_i)$,
    since any set of diameter $r$ in $2^\omega$ is contained in an open
    ball of the same diameter. Then
    \[ \sum_{i} g(\diam(U_i)) \geq \sum_{i} g(r_i) \geq \sum_i
    \mu(B_{r_i}(x_i)) \geq \mu(B)>0.\]
    So, in particular $H^g_{\delta}(B)\geq \mu(B)$ for any $\delta<r$, and so $H^g(B) \geq
    \mu(B)$.
  \end{proof}
  It remains to show that $\mu_{A,y}$
  satisfies the conditions of the claim. 
  Now we have
$\mu_{A,y}(B_{2^{-n}}(x)) = 2^{-n+|A\cap n|} < g(2^{-n})$ for all but
finitely many $n$, where the last
inequality follows from ($\dagger$).  

Finally, note that there are infinite sets $A$ satisfying ($\dagger$)
  since by assumption that $g \prec \id$, we have $\lim_{t\to 0^+}
  \frac{g(t)}{t}=\infty$.
\end{proof}

Next we show that if $\leq_Q$ is a locally countable Borel quasi-order and
$y$ is sufficiently generic, then $\mu_{A,y}$-a.e.\ $x \in [T_{A,y}]$ is not
$\leq_Q$-above any other element of $[T_{A,y}]$. So there is a
$\mu_{A,y}$-conull $\leq_Q$-antichain in $[T_{A,y}]$. 

\begin{lemma}\label{lem:comeager_good}
  Suppose $A \subset \omega$ is infinite, and
  $(f_i)_{i \in \omega}$ is a countable set of Borel functions on $\cantor$.
  Then for a comeager set of $y \colon 2^{< \omega} \to 2$, for
  $\mu_{A,y}$-a.e.\ $x \in [T_{A,y}]$, for all $i \in \omega$, if $f_i(x) \neq x$, then $f_i(x)
  \notin [T_{A,y}]$.
\end{lemma}
\begin{proof}
If $f_i(x) \neq x$, then there is some $s \in 2^{< \omega}$ so that $x
\supset s$ and $f_i(x) \nsupset s$. 
Fix such $i \in \omega$ and $s \in 2^{< \omega}$. It suffices to show 
that for comeagerly many $y$, the set of $x \in [T_{A,y}]$ such that $x
\supset s$, $f_i(x) \nsupset s$ and $f_i(x) \in [T_{A,y}]$ is
$\mu_{A,y}$-null. The argument
will be by showing that as we build a generic $y$, it is dense to halve the measure of $x \in
[T_{A,y}]$ so that $f_i(x) \in [T_{A,y}]$.

By definition of $T_{A,y}$, we have $f_i(x) \notin [T_{A,y}]$ is equivalent to
$(\exists n \in A) f_i(x)(n) \neq y(f_i(x) \restriction n)$. Define 
\[B_y = \{x \in
[T_{A,y}] \colon x \supset s \land f_i(x) \nsupset s \land (\forall
n)(n \in A \implies f_i(x)(n) = y (f_i(x) \restriction n))\}.\] Elements in $B_y$
are the ``bad'' elements of $[T_{A,y}]$ and we want to show that for comeagerly many $y$, $\mu_{A,y}(B_y) = 0$.

If $p$ is a function from $2^k \to 2$, then let $B_{y,p} = \{x \in
[T_{A,y}] \colon x \supset s \land f_i(x) \nsupset s \land (\forall n
\leq k)(n \in A \implies f_i(x)(n) = p(f_i(x) \restriction n)\}$. The
difference between $B_y$ and $B_{y,p}$ is that the last $y$ in the
definition of $B_y$ has become $p$ in $B_{y,p}$.
So, if
$p_k = y \restriction 2^k$, then $B_y = \biginters B_{y,p_k}$.
We claim that given any $p \from 2^k \to 2$, there is a dense set of $q \supset p$
such that for comeagerly many $y \in \nbhd{q}$, we have $\lambda(B_{y,q}) \leq
\frac{1}{2}
\lambda(B_{y,p})$. This claim implies that for comeagerly many $y$,
$\mu_{A,y}(B_y) = 0$, which will conclude the proof. 

Suppose $k' > k$ and $p' \from 2^{k'} \to 2$ extends $p$. We need to show that
there is a $q$ extending $p'$ so that for comeagerly many $y \in \nbhd{q}$, we
have $\lambda(B_{y,q}) \leq \frac{1}{2} \lambda(B_{y,p})$.
Suppose $n \in A$ is such that
$n > k'$. Let $B_{y,p}^{0} = \{x \in B_{y,p} \colon f_i(x)(n) = 0\}$ and
$B_{y,p}^{1} = \{x \in B_{y,p} \colon f_i(x)(n) = 1\}$, so $B_{y,p} =
B^{0}_{y,p} \disjointunion B^{1}_{y,p}$. Now consider $C = \{y \colon
\lambda(B^{0}_{y,p}) < \frac{1}{2}\lambda(B_{y,p})\}$. This set is analytic
and so it has the Baire property. First, consider the case that $C$ is nonmeager in $\nbhd{p'}$. So there is
some $q' \supset p'$ such that $C$ is comeager in $\nbhd{q'}$. We may assume
that $q' \from 2^m \to 2$ where $m > n$.
Let
\[q(t) = \begin{cases}
0 & \text{ if $|t| = n$ and $t \nsupset s$} \\
q'(t) & \text{ otherwise}
\end{cases}\]
Since $q(t) = q'(t)$ for all $t$ compatible with $s$, and since the set
$B_{y,p}^0$ only depends on the values of $y(t)$ such that $t$ is
compatible with $s$, we have that $C$ is also comeager in $\nbhd{q}$.
Finally, since $q(t) = 0$ if $|t| = n$ and $t \nsupset s$, we have that 
$B_{y,q} \subset B^{0}_{y,p}$, and so $\lambda(B_{p,q})
\leq \frac{1}{2} \lambda(B_{y,p})$.

If $C$ is meager in $\nbhd{p'}$, then the set $\{y \colon
\lambda(B^1_{y,p}) \leq \frac{1}{2} \lambda(B_{y,p})\}$ is comeager in
$\nbhd{p'}$ (and in particular it is nonmeager). The argument in this case
is identical to the above argument, just changing the roles of $0$ and $1$. 
This finishes the proof of the claim. 
\end{proof}

\begin{thm}\label{thm:main_thm}
  If $\leq_Q$ is a locally countable Borel quasi-order on $\cantor$, and $g$ is a
  gauge function such that $g\prec \id$, then there is a closed
  $\leq_Q$-antichain $B \subset \cantor$ with $H^g(B) > 0$.
\end{thm}
\begin{proof}
By Lusin-Novikov uniformization, fix countably many Borel functions $(f_i)$ generating $\leq_Q$ and $A \subset \omega$ that is
sufficiently sparse as in Lemma~\ref{lem:sparse_measure}. Let $y \colon 2^{< \omega} \to 2$
be such that for
$\mu_{A,y}$-a.e.\ $x \in [T_{A,y}]$, if $f_i(x) \neq x$, then $f_i(x)
\notin T_{A,y}$. Such a $y$ exists since there is a comeager set of such
$y$ by Lemma~\ref{lem:comeager_good}. Thus, there is a $\mu_{A,y}$-conull
set $C \subset [T_{A,y}]$ that forms a $\leq_Q$-antichain, so by inner regularity
of $\mu_{A,y}$ there is a closed set $B \subseteq C$
with $\mu_{A,y}(B) > 0$ that is a $\leq_Q$-antichain. By
Lemma~\ref{lem:sparse_measure} $H^g(B)>0$.

\end{proof}
\begin{proof}[Proof of Theorem~\ref{thm:bigthm}]
  Suppose $E$ is a countable Borel equivalence relation on $\cantor$. Then
  viewing $E$ as a locally countable Borel quasi-order, if $B \subset
  \cantor$ is a closed antichain for $E$ so that $H^g(B) > 0$, then $B$
  meets each $E$-class in at most one point, so $E \restriction B$ is
  smooth.

  To see the last part of the theorem, let $g_s = t^s$ be the gauge
  function defining the $s$-dimensional Hausdorff measure $H^s$. Choose a
  gauge function $g \prec \id$ so that $g_s \prec g$ for all $s \preceq 1$,
  for instance $g(t) = t^{1 - \frac{1}{t}}$. Then $H^g(B) > 0$ implies
  $H^s(B) > 0$ for all $s < 1$, so $\dim(B) = 1$.
\end{proof}

By an analogous argument to the proof of Theorem~\ref{thm:bigthm}, 
 we get a similar result for locally countable Borel quasi-orders.
\begin{corollary}\label{cor:quasi-orderantichain}
  Suppose that $Q$ is a locally countable Borel quasi-order on $2^\omega$. Then there is a closed $Q$-antichain of Hausdorff dimension $1$.
\end{corollary}

\section{Results on $\R^n$}

We can transfer all our results above from the space $\cantor$ to the space
$\R^n$. This is because there are Borel bijections between $\cantor$ and
$[0,1]$ which preserve the property of having positive gauge measure. To show
this we begin with a proposition about functions between gauge measures on different
metric spaces.

\begin{proposition}\label{prop:gauge_cover}
  Suppose $(X_1,d_1)$ and $(X_2,d_2)$
  are metric spaces, $g$ is a gauge function, and $h \from [0,\infty) \to
  [0,\infty)$ is continuous and strictly increasing with $h(0) = 0$.
  Suppose 
  $f \from X_1 \to X_2$ has the property that 
  for all sets $A \subset X_1$, $f(A)$ can be covered by at most $k$ sets
  of $d_2$-diameter $h(\diam_{d_1}(A))$. 
  Then for any $B$, we have
  $H^{g\circ h^{-1}}(f(B)) \leq k H^g(B)$.
\end{proposition}
\begin{proof}
  Given any cover $(U_i)$
  of $B \subset X_1$ by sets of diameters less than $\delta$, the sets $f(U_i)$ cover $f(B)$, and we can cover each
  set $f(U_i)$ by $k$ sets $V_{i,1}, \ldots V_{i,k}$ of diameter at most
  $h(\diam(U_i))$. Hence, $H^{g \circ h^{-1}}_{h(\delta)}(f(B)) \leq k
  H^g_\delta(B)$. The proposition follows by taking the limit as $\delta \to 0$ since $h$ is continuous on the
  right.
\end{proof}

We will mostly apply this proposition below with $h$ equal to the identity.
In this case, the statement of the proposition becomes the following:
suppose for all sets $A \subset X_1$, $f(A)$ can be covered by at most $k$
sets of $d_2$-diameter $\diam_{d_1}(A)$. Then for any $B$, we have
$H^{g}(f(B)) \leq k H^g(B)$.

Proposition~\ref{prop:gauge_cover} is related to a classical result in fractal geometry that relates the Hausdorff measures, and thus Hausdorff dimensions, of sets and their images along H\"older continuous functions: If $A\subseteq R^n$ is any set and $f: R^n\to R^m$ is H\"older continuous with exponent $\alpha\in \R^+$ and multiplicative constant $c$, then for any $s\in \R^+$, $H^{s/\alpha}(f(A))\leq c^{s/\alpha}H^s(A)$~\cite[Proposition 3.1]{F14}. However, the hypothesis of Proposition~\ref{prop:gauge_cover} can be satisfied by functions that are not H\"older continuous. A prime example of such functions are the bijections between $\cantor$ and $[0,1]^n$ we will construct now.

\begin{proposition}\label{prop:cantortounit}
  There is a Borel bijection $f \from \cantor \to [0,1]^n$ so that for all sets
  $A$, and all gauge functions $g$, $H^g(A) > 0$ if and only if $H^{g \circ
  h^{-1}}(f(A))
  > 0$ where $h = t^{1/n}$. 
\end{proposition}
\begin{proof}
  We begin by proving the case $n = 1$. Let $f \from \cantor \to [0,1]$ map
  each infinite binary sequence $x$ to the real number given by the binary
  expansion of $x$. Since the dyadic rationals have both a finite and an
  infinite binary expansion, this map is not injective. However, $f$ is a
  bijection between $\{x \in \cantor \colon x \text{ is not eventually
  constant}\}$ and the complement of the dyadic rationals. Both these sets
  are co-countable in $\cantor$ and $[0,1]$ respectively. Hence, their
  complements have $H^g$-measure $0$ for every gauge function $g$, and we
  can redefine $f$ on this countable set so that it is a bijection from
  $\cantor \to [0,1]$, and hence ignore these countable sets in what
  follows. 

  Now any set $A \subset \cantor$ can be covered by a basic open set of the
  same diameter. So suppose $s \in 2^{< \omega}$ is a finite binary
  sequence of length $n = |s|$. Then the basic open set $[s] = \{x \in
  \cantor \colon x \supset s\}$ of diameter $2^{-n}$ is
  mapped by $f$ to an interval of the form $(p/2^n,(p+1)/2^n)$, which also
  has diameter $2^{-n}$ in the Euclidean metric. So $f$ has the property
  that for all $B$, $H^g(f(B)) \leq H^g(B)$ by Proposition
  \ref{prop:gauge_cover} letting $h$ be the identity and $k$ being $1$.

  Now we argue similarly for $f^{-1}$. Any set in $[0,1]$ can be
  covered by a closed interval of the same
  diameter. Suppose $[a,b]$ is a closed interval. 
  Let $m$ be the integer so that
  $1/2^{m} < \diam([a,b]) \leq 1/2^{m-1}$. There is a unique
  dyadic rational of the form $p/2^m$ in $(a,b)$ where $p$ is an integer.
  Hence, $[a,b] \subset [p/2^{m} - 1/2^{m-1},p/2^m + 1/2^{m-1}]$. So $[a,b]$ is covered by
  four dyadic intervals of length $1/2^m$: $[p/2^m - 1/2^{m-1},p/2^m -
  1/2^{m}], \ldots, [p/2^{m-1} + 1/2^m, p/2^m + 1/2^{m-1}]$. All of these
  intervals are the images of basic open sets in $\cantor$ of diameter 
  $1/2^m$ which is less than $\diam([a,b])$. So for any set $B$, we
  have $H^g(B) \leq 4 H^g(f(B))$ by Proposition~\ref{prop:gauge_cover}
  applied to $f^{-1}$. So combining with the above paragraph, we have that
  for all $B$, $\frac{1}{4} H^g(B) \leq H^g(f(B)) \leq H^g(B)$.

  Now we prove the case $n > 1$. Let $d_\infty$ be the metric on $(\cantor)^{n}$ defined by $d_\infty((x_1,
  \ldots, x_n),(y_1, \ldots, y_n)) = \sup_i d(x_i,y_i)$ where $d$ is the
  usual metric on $\cantor$. 
  Consider the function $j_n \from (\cantor,d) \to ((\cantor)^{n},d_\infty)$
  defined by $j(x) = (y_1, \ldots, y_n)$ where $y_i(j) = x(jn + i)$ so,
  $y_i$ is all the bits of $x$ that are $i$ mod $n$ in order. Then if
  $d(x,y) = 2^{-k}$, then $d(j(x),j(y)) = 2^{-\lfloor k/n \rfloor}$. 
  So using $h(t) = t^{1/n}$ and Proposition~\ref{prop:gauge_cover} on $j_n$ and
  $j_n^{-1}$, we conclude there are constants $k_1$ and $k_2$ so that
  $k_1 H^g(B) < H^{g \circ h}(f(B)) < k_2 H^g(B)$. 

  Finally, let $f_n \from
  \cantor \to [0,1]^n$ be defined by $f_n(x) = (f(y_1), \ldots, f(y_n))$
  where $j_n(x) = (y_1, \ldots, y_n)$, and $f$ is the function from the
  case $n = 1$ defined above. 
  Then apply Proposition~\ref{prop:gauge_cover} and note that
  if $d_\infty$ is
  the sup metric on $[0,1]^n$, and $d$ is the usual Euclidean metric on
  $[0,1]^n$, then any set $A \subset [0,1]^n$ of $d$-diameter $r$
  has $d_\infty$-diameter at most $r$. Conversely, there is a constant
  $c_n$ so that any set $A \subset [0,1]^n$ of $d_\infty$ diameter
  $r$ can be covered by $c_n$ sets of $d$-diameter $r$.
\end{proof}

Now we can obtain a version of Proposition~\ref{thm:bigthm} for $\R^n$.
\begin{corollary}\label{cor:unitinterval}
Suppose $g \from [0,\infty) \to [0,\infty)$ is a gauge function of lower order than $t \mapsto t^n$
and that $E$ is a countable Borel equivalence relation on $\R^n$. 
Then there is a closed set $A \subset [0,1]$ such that $E \restriction A$ is
smooth, and $H^g(A) > 0$. In particular, there is a closed set of Hausdorff
dimension $n$ such that $E \restriction A$ is smooth.
\end{corollary}
\begin{proof}
  Let $g'(t) = g(t^{1/n})$. Note that $g'$ has lower order than the
  identity if and only if $g$ has lower order than $t^n$. Let $f \from \cantor \to [0,1]^n$ be the Borel bijection from
  Proposition~\ref{prop:cantortounit}.
  Given $E$ on $\R^n$, define $E'$ on $\cantor$ by $x \mathrel{E'} y$ if 
  $f(x) \mathrel{E} f(y)$. We can apply Theorem~\ref{thm:bigthm} to $E'$ to
  obtain a closed set $A'$ with $H^{g'}(A') > 0$ and such that $E' \restriction A'$ is smooth. Now $f(A') \subset [0,1]$ has positive $H^g$ measure by
  Proposition~\ref{prop:cantortounit} and is Borel since an injective image
  of a Borel set under a Borel function is Borel. 
  So $E \restriction f(A')$ is smooth since $f$ is a bijection and $E'
  \restriction A'$ is smooth. To finish, let $A \subset f(A')$
  be closed with $H^{g}(A) > 0$.

  To see the last part of the corollary, recall that Theorem~\ref{thm:bigthm} 
  allows us to take $A'$ with $\dim(A')=1$ such that $E'\restriction A'$ is
  smooth. Thus, by the above arguments we can get an $A \subset f(A')$ with
  $\dim(A)=n$ so that $E\restriction A$ is smooth.
\end{proof}
%
We finish by noting that the same arguments used to prove
Corollaries~\ref{cor:unitinterval} can be used to obtain
an analogue of this result for locally countable Borel quasi-orders.
\begin{corollary}
  Suppose that $Q$ is a locally countable Borel quasi-order on $\R^n$. Then there is a closed $Q$-antichain of Hausdorff dimension $n$.
\end{corollary}

\end{document}